\documentclass[preprint,moreauthors,11pt]{elsarticle}

\usepackage{amssymb}

\usepackage{lineno,color}
\usepackage{amsmath, amsthm}

\theoremstyle{definition}
  \newtheorem{defn}{Definition}[section]
  
\theoremstyle{plain}
  \newtheorem{lem}[defn]{Lemma}
  \newtheorem{thm}[defn]{Theorem}
  
  \newtheorem{cor}[defn]{Corollary}

\numberwithin{equation}{section}

\usepackage{lineno,hyperref}

\journal{j}

\begin{document}
\begin{frontmatter}

\newtheorem{theorem}{Theorem}[section]
\newtheorem{lemma}[theorem]{Lemma}
\newtheorem{proposition}[theorem]{Proposition}
\newtheorem{corollary}[theorem]{Corollary}
\theoremstyle{definition}
\newtheorem{definition}[theorem]{Definition}
\newtheorem{example}[theorem]{Example}
\newtheorem{xca}[theorem]{Exercise}
\newtheorem{problem}[theorem]{Problem}
\theoremstyle{remark}
\newtheorem{remark}[theorem]{Remark}
\newtheorem{note}[theorem]{Note}
\numberwithin{equation}{section}

\title{Hyers-Ulam stability of homomorphisms and $\ast$-homomorphisms on Hilbert $C^{*}$-modules}

\author[a]{Sajjad Khan}
\address[a]{Department of Mathematics, Hanyang University, Seoul 04763, Korea}
\ead{sajjadafridi@hanyang.ac.kr}

\author[b]{Choonkil Park\corref{mycorrespondingauthor}}
\address[b]{Research Institute for Convergence of Basic Sciences, Hanyang University,  Seoul 04763, Korea}
\ead{baak@hanyang.ac.kr}
\cortext[mycorrespondingauthor]{Corresponding author}

\begin{abstract}
In this paper, we introduce the idea of $\ast$-homomorphism on a Hilbert $C^{*}$-module. Furthermore, we prove the Hyers-Ulam stability of homomorphisms and $\ast$-homomorphisms on Hilbert $C^{*}$-modules using the fixed point method.
\end{abstract}

\end{frontmatter}

\section{Introduction and preliminaries} 

The stability problem of functional equations was arose from a question posed by Ulam \cite{Ulam} 
in 1940.
Specifically, the question was: Let $(G,\ast)$ be a group and $(G^{\diamond},\diamond, d^{\diamond})$ be a metric group with the metric $d^{\diamond}$. For any $\epsilon >0$, does there exist a $\delta >0$ such that if $f:G\rightarrow G^{\diamond}$ satisfies $d(f(x \ast y),f(x)\diamond f(y))< \delta$ for all $x,y \in G$, then a homomorphism $h:G \rightarrow G^{\diamond}$ exists with $d^{\diamond}(f(x),h(x))<\epsilon$ for all $x \in G$? If such $h$ exists, then the functional equation $h(x \ast y)=h(x)\diamond h(y)$ is stable. In 1941, Hyers \cite{Hyers} provided the first affirmative answer for approximately additive mappings in the context of Banach spaces. In 1978, Rassias \cite{Rassias 1978} generalized the Hyers' result for linear mappings. G\v{a}vruta \cite{Gavruta} introduced a general control function instead of the Cauchy difference introduced in \cite{Rassias 1978}, and further extended the Rassias' results. Rassias \cite{Rassias 1990} posed the question in 1990 about whether a similar theorem could be established for 
$p\geq 1$. Gajda \cite{Gajda} provided an affirmative answer to this question for 
$p > 1$, following the same approach as in \cite{Rassias 1978} by Rassias. Gajda \cite{Gajda}, along with Rassias and \v{S}emrl \cite{Rassias 1993}, showed that the Rassias' results cannot be established for $p=1$.\

In 1991, Baker \cite{Baker} used the fixed point method for the first time to study the stability of functional equations. Then, Radu \cite{Radu} and Cadariu and Radu \cite{Cadariu1} studied the stability of functional equations via fixed point method. In 1996, Isac and Rassias \cite{Isac} were the first to apply the stability theory of functional equations to establish new fixed point theorems and their applications. Using fixed point methods, many authors have extensively investigated the stability problem of various functional equations (see, for example \cite{Cadariu2, Siriluk, Park 2017}).

Park \cite{Park rho 1, Park rho 2} introduced additive $\rho$-functional inequalities and established the Hyers-Ulam stability of these inequalities in both Banach spaces and non-Archimedean Banach spaces. The stability problems of various functional equations, functional inequalities, and differential equations have been thoroughly studied by numerous researchers (see \cite{Sajjad1, Sajjad2, Park bi-derivation}).

In mathematics, the notion of a Hilbert $C^*$-module generalizes the concept of a Hilbert space by replacing the field of complex numbers with a general $C^{*}$-algebra. This concept was first introduced by Kaplansky in  \cite{kaplansky}. The stability of various mappings on Hilbert $C^{*}$-modules has been investigated by several mathematicians (see \cite{Ali, Ismail}).

Let $\mathcal{A}$ be a $C^{*}$-algebra and $\Xi$ be a complex linear space equipped with a compatible left $\mathcal{A}$-module action (i.e., $\lambda(a x)=(\lambda a)x=a(\lambda x)$ for $x \in \Xi$, $a \in \mathcal{A}$, $\lambda\in\mathbb{C}$). The space $\Xi$ is called a (left) pre-Hilbert $\mathcal{A}$-module or (left) inner product $\mathcal{A}$-module if there exists a mapping $\langle\cdot, \cdot\rangle: \Xi \times \Xi \rightarrow \mathcal{A}$, satisfying the following conditions: For all $x,y,z \in \Xi, \alpha,\beta \in \mathbb{C}, a \in \mathcal{A},$
\begin{itemize}
\item[(i)] $\langle \alpha x + \beta y,z  \rangle= \alpha \langle x,z \rangle + \beta \langle y,z \rangle$;
\item[(ii)] $\langle a x,y \rangle=a \langle x,y \rangle $;
\item[(iii)] $\langle x,y \rangle^{*} =\langle y,x \rangle$;
\item[(iv)] $\langle x,x \rangle\geq 0$ and $\langle x,x\rangle=0$ if and only if $x=0.$
\end{itemize}
The (left) pre-Hilbert space $\Xi$ is considered a (left) Hilbert $\mathcal{A}$-module or a (left) Hilbert $C^{*}$-module over the $C^{*}$-algebra $\mathcal{A}$ if it is complete with the induced norm $\|x\|:=\|\langle x , x\rangle\|^\frac{1}{2}$ \cite{Lance}. Right Hilbert $C^{*}$-modules can be defined in a similar way. In this paper, we will simply refer to a (left) Hilbert $C^*$-module over the $C^{*}$-algebra $\mathcal{A}$ as a Hilbert $C^*$-module. As an example, it is easy to see that every complex Hilbert space $\Xi$ is a Hilbert $C^*$-module over the $C^{*}$-algebra $\mathbb{C}$, with its inner product. Similarly, every $C^*$-algebra $\mathcal{A}$ can be regarded as a Hilbert $C^*$-module where, 
$\langle x,y\rangle = x y^{*}$, for $x,y \in \mathcal{A}$.

Khan \textit{et al.} \cite{Sajjad3} introduced the idea of fuzzy Hilbert $C^{*}$-module homomorphisms and fuzzy Hilbert $C^{*}$-module derivations and studied their Hyers-Ulam stability while in \cite{Sajjad4} they introduced Jordan homomorphisms and Jordan $\ast$-homomorphisms in Hilbert $C^{*}$-modules and discussed their Hyers-Ulam stability.

 In this paper, we introduce Hilbert $C^{*}$-module $\ast$-homomorphisms and prove the Hyers-Ulam stability of the Hilbert $C^{*}$-module homomorphisms and Hilbert $C^{*}$-module $\ast$-homomorphisms. It is observed that the definition of a Hilbert $C^{*}$-module homomorphism first appeared in \cite{Homomorphism definition}. 

\begin{defn} \cite{Homomorphism definition}
Let $\mathcal{A}$ and $\mathcal{B}$ be $C^{*}$-algebras, $\Xi$ be a Hilbert $\mathcal{A}$-module and $\nabla$ be a Hilbert $\mathcal{B}$-module. A mapping $H:\Xi \rightarrow \nabla$ is called a Hilbert $C^{*}$-module homomorphism if $H$ is $\mathbb{C}$-linear and satisfies
\begin{equation}\label{1.1}
H\left( \langle x,y\rangle z\right) =\langle H(x),H(y)\rangle H(z)
\end{equation}
for all $x,y,z \in \Xi $.
\end{defn}

\begin{defn}\label{1.2}
A Hilbert $C^{*}$-module homomorphism $H: \Xi \rightarrow \nabla $ is called a Hilbert $C^{*}$-module $\ast$-homomorphism if 
\begin{equation}
H\left( \langle x,y\rangle ^{*} z\right) =\langle H(x),H(y)\rangle ^{*} H(z)
\end{equation}
for all $x,y,z \in \Xi $.
\end{defn}

We recall a fundamental result in the fixed point theory. 
\begin{thm}\label{FPA}\cite{Cadariu1}
 Let (X,d) be a complete generalized metric space
and $J: X\rightarrow X$ be a strictly contractive mapping with
Lipschitz constant $L<1$. Then, for all $x\in X$, either
\begin{equation*}
d(J^n x,J^{n+1}x)=\infty
\end{equation*}
for all nonnegative integers $n$ or there exists a positive
integer $n_0$ such that

$(1)$ $d(J^n x,J^{n+1}x)<\infty$ for all $n\geq n_0$;

$(2)$ the sequence $\{J^n x\}$ converges to a fixed point $y^*$ of
$J$;

$(3)$ $y^*$ is the unique fixed point of $J$ in the set $Y=\{y\in
X : d(J^{n_0}x,y)<\infty\}$;

$(4)$ $d(y,y^*)\leq \frac{1}{1-L}d(y,Jy)$ for all $y\in Y$.
\end{thm}

Throughout the paper, we will denote $\Xi$ as a Hilbert $\mathcal{A}$-module and $\nabla$ as a Hilbert $\mathcal{B}$-module. The paper is organized as follows: In Section 2, we prove the Hyers-Ulam stability of Hilbert $C^{*}$-module homomorphisms. In Section 3, we establish the same results for Hilbert $C^{*}$-module $\ast$-homomorphisms. 
%




\section{Stability of Hilbert $C^{*}$-module homomorphisms}
We will need the following lemma for our main results.

\begin{lem}\label{c-linear}\cite{c-linear}
Let $X$ and $Y$ be linear spaces and $f:X \rightarrow Y$ be an additive mapping such that $f(\mu x)=\mu f(x)$ for all $x \in X$ and any $\mu \in \mathbb{T}=\lbrace z \in \mathbb{C} : \vert z\vert =1\rbrace$. Then, the mapping f is $\mathbb{C}$-linear.
\end{lem}

\begin{thm}\label{theorem 1}
Let $f:\Xi \rightarrow \nabla$ be a mapping for which there exists a function  $\varphi: \Xi^{3}\rightarrow [0,\infty)$ such that 
\begin{equation}\label{2.1}
\vert\vert f(\mu x +y)-\mu f(x)-f(y)\vert\vert \leq \varphi(x,y,0),
\end{equation}
\begin{equation}\label{2.2}
\| f(\langle x,y\rangle z)-\langle f(x), f(y) \rangle f(z) \| \leq \varphi(x,y,z)
\end{equation}
for all $x,y,z \in \Xi$ and all $\mu \in \mathbb{T}$. If there exists  $0\leq L < 1$ such that $\varphi (x,y,z)\leq 2 L \varphi (\frac{x}{2},\frac{y}{2},\frac{z}{2})$ for all $x,y,z \in \Xi $, then there exists a unique Hilbert $C^{*}$-module homomorphism $H:\Xi \rightarrow \nabla$ such that 
\begin{equation}\label{2.3}
\| f(x)- H(x) \| \leq \dfrac{1}{2-2L}\varphi(x,x,0)
\end{equation}
for all $x \in \Xi$.
\end{thm}

\begin{proof}
Consider $\mathcal{S}=\lbrace g:\Xi\rightarrow \nabla \rbrace$, and define the generalized metric on $\mathcal{S}$:
\begin{equation}\label{2.4}
d(g,h):= \text{inf} \lbrace c \in (0,\infty):\|g(x)-h(x)\|\leq c \varphi(x,x,0), \forall x \in \Xi\rbrace.
\end{equation}
It is easy to show that $(\mathcal{S},d)$ is a complete metric space (see \cite{m-r}). 

Now, consider the linear mapping $J: \mathcal{S}\rightarrow \mathcal{S}$
\begin{equation}\nonumber
(Jg)(x):= \dfrac{1}{2}g(2x)
\end{equation}
for all $x\in \Xi$.\\
By \cite[Theorem 3.1]{Cadariu1}, $J$ is a strictly contractive mapping.

Letting $x=y$, $\mu =1$ in (\ref{2.1}), we obtain
\begin{equation}\label{2.5}
\left \| f(2x)-2f(x) \right \|\leq \varphi(x,x,0)
\end{equation}
for all $x \in \Xi$. This implies that 
\begin{equation}\nonumber
\left \| f(x)- \dfrac{1}{2}f(2x) \right \| \leq\frac{1}{2}\varphi(x,x,0)
\end{equation}
for all $x \in \Xi$. Thus $d(f,Jf)\leq \frac{1}{2}$.

By Theorem \ref{FPA}, there exists a mapping $H:\Xi \rightarrow \nabla$ which is a fixed point of $J$ such that 
\begin{equation}\label{2.6}
H(2x)=2H(x)
\end{equation}
for all $x \in \Xi$ and $\displaystyle{\lim_{n\to\infty}} d(J^{n}f,H)=0$. This implies that 
\begin{equation}\label{2.7}
 \lim_{n\to\infty}\dfrac{f(2^{n}x)}{2^{n}}=H(x) 
\end{equation}
for all $x \in \Xi.$ Notice that  $H$ is the unique fixed point of $J$ in the set 
\begin{equation}\nonumber
U=\lbrace g \in \mathcal{S}: d(g,f)<\infty\rbrace.
\end{equation}
This implies that $H$ is the unique fixed point satisfying (\ref{2.6}) such that there exists $c \in (0,\infty)$ satisfying 
\begin{equation}\nonumber
\| f(x) - H(x)\|\leq c \varphi(x,x,0)
\end{equation}
for all $x \in \Xi$.\\
Applying Theorem \ref{FPA} once again, we obtain $d(f,H)\leq \frac{1}{1-L}d(f,Jf)\leq \frac{1}{2-2L}.$ This implies 
\begin{equation}\nonumber
\|f(x)-H(x)\|\leq \frac{1}{2-2L} \varphi(x,x,0)
\end{equation}
for all $x\in \Xi$ and (\ref{2.3}) holds.

Now we show that $H$ is $\mathbb{C}$-linear. From the assumption $\varphi(x,y,z) \leq 2L \varphi(\frac{x}{2},\frac{y}{2},\frac{z}{2})$ for all $x,y,z \in \Xi$, it follows that
\begin{equation}\label{2.8}
\lim_{n\to\infty} \frac{1}{2^{n}}\varphi (2^{n}x,2^{n}y,2^{n}z)\leq \lim_{n\to\infty} \frac{2^{n}L^{n}}{2^{n}}\varphi (x,y,z)=0.
\end{equation}
In (\ref{2.1}), replacing $x$ by $2^{n}x$ and $y$ by $2^{n}y$, respectively, we have
\begin{equation}\nonumber
\|f(\mu 2^{n}x+2^{n}y)-\mu f(2^{n}x)-f(2^{n}y)\|\leq \varphi(2^{n}x,2^{n}y,0).
\end{equation}
Taking the limit as $n \rightarrow \infty$, using (\ref{2.7}) and (\ref{2.8}), we obtain 
\begin{equation}\label{2.9}
H(\mu x+y)=\mu H(x)+H(y)
\end{equation}
for all $\mu \in \mathbb{T}$ and all $x,y \in \Xi.$ The case  $\mu=1$ in (\ref{2.9}) implies that $H$ is additive and so $H(0)=0$, together with these, Lemma \ref{c-linear} implies that $H$ is $\mathbb{C}$-linear.

Putting $x,y,z$ as $2^{n}x,2^{n}y,2^{n}z$, respectively, in (\ref{2.2}), we obtain 
\begin{equation}\nonumber 
\|f(\left\langle  2^{n}x,2^{n}y \right\rangle  2^{n}z)- \left\langle  f(2^{n}x),f(2^{n}y)\right\rangle  f(2^{n}z)\|\leq \varphi(2^{n}x,2^{n}y,2^{n}z).
\end{equation}
This further implies that 
\begin{align*}\nonumber
\left \| 2^{-3n}f(2^{3n}\left\langle  x,y\right\rangle  z)-\left\langle  \frac{f(2^{n}x)}{2^{n}},\frac{f(2^{n}y)}{2^{n}}\right\rangle  \frac{f(2^{n}z)}{2^{n}} \right \| & \leq 2^{-3n}\varphi(2^{n}x,2^{n}y,2^{n}z)\\
& \leq  2^{-3n}8^{n}L^{n}\varphi(x,y,z)
\end{align*}
for all $x,y,z \in \Xi$, since $\varphi(x,y,z)\leq 2L \varphi(\frac{x}{2},\frac{y}{2},\frac{z}{2}) \leq 8L \varphi(\frac{x}{2},\frac{y}{2},\frac{z}{2})$.\\
Taking the limit as $n\rightarrow \infty $, applying (\ref{2.7}) and (\ref{2.8}), we get 
\begin{equation}\nonumber
H(\left\langle  x,y\right\rangle  z) = \left\langle H(x), H(y)\right\rangle H(z)
\end{equation}
for all $x,y,z \in \Xi$. Thus, $H$ is a Hilbert $C^{*}$-module homomorphism. 
\end{proof}

The following corollary gives us the Hyers-Ulam-Rassias stability of homomorphisms in $\Xi$.

\begin{cor}
Let $p\in [0,1)$, $\theta$ be a nonnegative real number and $f:\Xi \rightarrow \nabla $ be a mapping such that 
\begin{equation}\label{2.10}
\| f(\mu x +y)-\mu f(x)-f(y)\|  \leq \theta \left( \|x\|^{p}+\|y\|^{p}\right), 
\end{equation}
\begin{equation}\label{2.11}
\| f(\left\langle  x,y\right\rangle z)- \left\langle f(x),f(y)\right\rangle f(z) \|  \leq \theta \left( \|x\|^{p}+\|y\|^{p}+\|z\|^{p} \right) 
\end{equation}
for all $x,y,z \in \Xi$ and all $\mu \in \mathbb{T}$. Then there exists a unique Hilbert $C^{*}$-module homomorphism $H:\Xi \rightarrow \nabla $ such that 
\begin{equation}
\| f(x)- H(x) \| \leq \frac{2 \theta}{2-2^{p}}\|x\|^{p}
\end{equation}
for all $x \in \Xi$.
\end{cor}

\begin{proof}
Take $\varphi(x,y,z):=\theta\left( \|x\|^{p}+\|y\|^{p}+\|z\|^{p}\right)$ and $L=2^{p-1}$. Then the proof follows from Theorem \ref{theorem 1}.
\end{proof}

\begin{thm}\label{theorem 2}
Let $f:\Xi \rightarrow\nabla $ be a mapping for which there exists a function $\varphi: \Xi^{3}\rightarrow [0,\infty)$ satisfying (\ref{2.1}) and (\ref{2.2}) for all $x,y,z \in \Xi$ and all $\mu \in \mathbb{T}$. 
If there exists  $0 \leq L < 1$ such that $\varphi (x,y,z)\leq \frac{L}{8} \varphi (2x,2y,2z)$ for all $x,y,z \in \Xi$, then there exists a unique Hilbert $C^{*}$-module homomorphism $H:\Xi\rightarrow \nabla$ such that 
\begin{equation}\label{2.13}
\| f(x)- H(x) \| \leq \dfrac{L}{8-8L}\varphi(x,x,0)
\end{equation}
for all $x \in \Xi$.
\end{thm}

\begin{proof}
Consider the linear mapping $J: \mathcal{S}\rightarrow \mathcal{S}$
\begin{equation}\nonumber
Jg\left( x \right) := 2g\left( \frac{1}{2}x\right) 
\end{equation}
for all $x\in \Xi$.\\
It follows from (\ref{2.5}) that 
\begin{equation}\nonumber
\left \| f(2x)-2f\left( \frac{x}{2}\right)  \right \| \leq \varphi\left( \frac{x}{2},\frac{x}{2},0\right)  \leq \frac{L}{8} \varphi(x,x,0)
\end{equation}
for all $x \in \Xi$. By (\ref{2.4}), we have $d(f,Jf)\leq \frac{L}{8}$.

By Theorem \ref{FPA}, there exists a mapping $H:\Xi \rightarrow \nabla$ which is a fixed point of $J$ such that 
\begin{equation}\label{2.14}
H\left( \frac{1}{2}x\right) =\frac{1}{2}H(x)
\end{equation}
for all $x \in \Xi$. Also $\displaystyle{\lim_{n\to\infty}}d(J^{n}f,H)=0$. This implies that 
\begin{equation}\label{2.15}
 \lim_{n\to\infty}2^{n}f\left( \frac{x}{2^{n}}\right) =H\left( x\right) 
\end{equation}
for all $x \in \Xi.$ Notice that  $H$ is the unique fixed point of $J$ in the set 
\begin{equation}\nonumber
U=\lbrace g \in \mathcal{S}: d(g,f)<\infty\rbrace.
\end{equation}
This implies that $H$ is the unique fixed point satisfying (\ref{2.14}) such that there exists $c \in (0,\infty)$ satisfying 
\begin{equation}
\| f(x) - H(x)\|\leq c \varphi(x,x,0)
\end{equation}
for all $x \in \Xi$.\\
Applying Theorem \ref{FPA} once again, we obtain $d(f,H)\leq \frac{1}{1-L}d(f,Jf)\leq \frac{L}{8-8L}.$ This implies 
\begin{equation}\nonumber
 \|f(x)-H(x) \|\leq \frac{L}{8-8L} \varphi(x,x,0)
\end{equation}
for all $x\in \Xi$ and (\ref{2.13}) holds.

From the assumption $\varphi(x,y,z) \leq \frac{L}{8}\varphi(2x,2y,2z)\leq\frac{L}{2}\varphi(2x,2y,2z) $ for all $x,y,z \in \Xi$, it follows that 
\begin{equation}\label{2.17}
\lim_{n\to\infty} 2^{n}\varphi \left( \frac{x}{2^{n}},\frac{y}{2^{n}},\frac{z}{2^{n}}\right) \leq \lim_{n\to\infty} \frac{2^{n}L^{n}}{2^{n}}\varphi (x,y,z)=0.
\end{equation}
In (\ref{2.1}), replacing $x$ by $\frac{x}{2^{n}}$, $y$ by $\frac{y}{2^{n}}$, multiplying both side by $2^{n}$, letting $n \rightarrow \infty$, using (\ref{2.15}) and (\ref{2.17}), we obtain
\begin{equation}\label{2.18}
H(\mu x+y)=\mu H(x)+H(y)
\end{equation}
for all $x,y \in \Xi$ and all $\mu \in \mathbb{T}$. The case $\mu =1$ in (\ref{2.18}) implies that $H$ is additive and by Lemma \ref{c-linear}, $H$ is $\mathbb{C}$-linear.

Replacing $x,y,z$ with $\frac{x}{2^{n}},\frac{y}{2^{n}},\frac{z}{2^{n}}$, respectively, in (\ref{2.2}), we obtain 
\begin{align*}\nonumber
\left \| 2^{3n} f \left(\left\langle  \frac{x}{2^{n}},\frac{y}{2^{n}} \right\rangle \frac{z}{2^{n}} \right) - \left\langle 2^{n}f \left( \frac{x}{2^{n}} \right) ,2^{n} f \left( \frac{y}{2^{n}} \right) \right\rangle 2^{n} f \left( \frac{z}{2^{n}}\right) \right \|  & \leq 2^{3n} \varphi \left( \frac{x}{2^{n}},\frac{y}{2^{n}},\frac{z}{2^{n}}\right)\\
& \leq 2^{3n}2^{-3n}L^{n} \varphi \left( x,y,z \right)
\end{align*}
for all $x,y,x \in \Xi$. Taking $n \rightarrow \infty$, using (\ref{2.15}) and (\ref{2.17}), we obtain 
\begin{equation}\nonumber
H(\left\langle x,y\right\rangle z)=\left\langle H(x),H(y)\right\rangle H(z)
\end{equation}
for all $x,y,z \in \Xi$. Thus $H$ is a Hilbert $C^{*}$-module homomorphism. 
\end{proof}

\begin{cor}
Let $p>3$, $\theta$ be a nonnegative real number and $f:\Xi \rightarrow \nabla$ be a mapping satisfying (\ref{2.10}) and (\ref{2.11}). Then there exists a unique Hilbert $C^{*}$-module homomorphism $H: \Xi \rightarrow \nabla$ such that 
\begin{equation}\nonumber
\| f(x)- H(x) \| \leq \dfrac{2\theta}{2^{p}-2^{3}}\|x\|^{p}
\end{equation}
for all $x \in \Xi$.
\end{cor}

\begin{proof}
Take $\varphi(x,y,z):=\theta\left( \|x\|^{p}+\|y\|^{p}+\|z\|^{p}\right)$ and $L=2^{3-p}$. Then the proof follows from Theorem \ref{theorem 2}.
\end{proof}

\section{Stability of Hilbert $C^{*}$-module $\ast$-homomorphisms} 

In this section, we establish the Hyers-Ulam stability of Hilbert $C^{*}$-module $\ast$-homomorphisms. 

\begin{theorem}\label{theorem 3}
Let $f:\Xi \rightarrow \nabla$ be a mapping for which there exists a function $\varphi: \Xi^{3}\rightarrow [0,\infty)$ satisfying (\ref{2.1}) and (\ref{2.2}) and 
\begin{equation}\label{3.1}
\| f(\left\langle x,y\right\rangle ^{*}z)- \left\langle f(x),f(y)\right\rangle ^{*}f(z) \| \leq \varphi (x,y,z) 
\end{equation}
for all $x,y,z \in \Xi$. If there exists  $0\leq L < 1$ such that $\varphi (x,y,z)\leq 2 L \varphi (\frac{x}{2},\frac{y}{2},\frac{z}{2})$ for all $x,y,z \in \Xi$, then there exists a unique Hilbert $C^{*}$-module $*$-homomorphism $H:\Xi \rightarrow \nabla $ such that 
\begin{equation}\label{3.2}
\| f(x)- H(x) \| \leq \dfrac{1}{2-2L}\varphi(x,x,0)
\end{equation}
for all $x \in \Xi$.
\end{theorem}
\begin{proof}
From Theorem \ref{theorem 1}, there exists a unique Hilbert $C^{*}$-module homomorphism $H$ satisfying (\ref{3.2}) and defined as 
\begin{equation}\label{3.3}
\lim_{n\to\infty}\frac{f(2^{n}x)}{2^{n}} =H\left( x \right) 
\end{equation}
for all $x \in \Xi$.

Replacing $x,y,z$ by $2^{n}x,2^{n}y,2^{z}$, respectively, in (\ref{3.1}), we obtain 
\begin{equation}\nonumber
\| f (\left\langle 2^{n}x,2^{n}y\right\rangle^{*} 2^{n}z) - \left\langle f(2^{n}x),f(2^{n}y)\right\rangle^{*} f(2^{n}z)\| \leq \varphi (2^{n}x,2^{n}y,2^{n}z).
\end{equation}
This further implies that 
\begin{align*}\nonumber 
\left \|2^{-3n} f (2^{3n}\left\langle x,y\right\rangle^{*} z) - \left\langle \frac{f(2^{n}x)}{2^{n}},\frac{f(2^{n}y)}{2^{n}}\right\rangle^{*} \frac{f(2^{n}z)}{2^{n}} \right \| & \leq 2^{-3n} \varphi (2^{n}x,2^{n}y,2^{n}z)\\
& \leq 2^{-3n}2^{3n}L^{n} \varphi (x,y,z)
\end{align*}
for all $x,y,z\in \Xi$. Taking the limit as $n \rightarrow \infty$, using (\ref{2.15}) and (\ref{2.17}), we obtain
\begin{equation}\nonumber 
H(\left\langle x,y\right\rangle ^{*}z)=\left\langle H(x),H(y)\right\rangle^{*} H(z)
\end{equation}
for all $x,y,z \in \Xi$. So $H$ is a Hilbert $C^{*}$-module $\ast$-homomorphism.
\end{proof}

The following corollary gives us the Hyers-Ulam-Rassias stability of Hilbert $C^{*}$-module $\ast$-homomorphisms.

\begin{cor}
Let $p\in [0,1)$, $\theta$ be a nonnegative real number and $f:\Xi \rightarrow \nabla$ be a mapping satisfying (\ref{2.10}), (\ref{2.11}) and 
\begin{equation}\label{3.4}
\| f(\left\langle x,y\right\rangle ^{*}z)- \left\langle f(x),f(y)\right\rangle ^{*}f(z) \| \leq \theta (\|x\|^{p}+\|y\|^{p}+\|z\|^{p})
\end{equation}
for all $x,y,z \in \Xi$. Then there exists a unique Hilbert $C^{*}$-module $\ast$-homomorphism $H:\Xi \rightarrow \nabla$ such that 
\begin{equation}\nonumber
\| f(x)- H(x) \| \leq \frac{2 \theta}{2-2^{p}}\|x\|^{p}
\end{equation}
for all $x \in \Xi$.
\end{cor}

\begin{proof}
The proof follows from Theorem \ref{theorem 3} by taking $\varphi (x,y,z)=\theta (\|x\|^{p}+\|y\|^{p}+\|z\|^{p})$ and $L=2^{p-1}$. 
\end{proof}

\begin{theorem}\label{theorem 4}
Let $f:\Xi \rightarrow \nabla$ be a mapping for which there exists a function $\varphi: \Xi^{3}\rightarrow [0,\infty)$ satisfying (\ref{2.1}) and (\ref{2.2}) and (\ref{3.1}). If there exists  $0\leq L < 1$ such that $\varphi (x,y,z)\leq \frac{L}{8} \varphi (2x,2y,2z)$ for all $x,y,z \in \Xi$, then there exists a unique Hilbert $C^{*}$-module $\ast$-homomorphism $H:\Xi \rightarrow \nabla$ such that 
\begin{equation}\label{3.5}
\| f(x)- H(x) \| \leq \dfrac{L}{8-8L}\varphi(x,x,0)
\end{equation}
for all $x \in \Xi$.
\end{theorem}
\begin{proof}
By Theorems \ref{theorem 2}, there exists a unique Hilbert $C^{*}$-module homomorphism $H:\Xi\rightarrow \nabla$ satisfying (\ref{3.5}) and defined as
\begin{equation}\label{3.6}
\lim_{n\to\infty}2^{n}f\left( \frac{x}{2^{n}}\right)  =H \left( x \right) 
\end{equation}
for all $x \in \Xi$.\\
Replacing $x,y,z$ by $\frac{x}{2^{n}},\frac{y}{2^{n}},\frac{z}{2^{n}}$, respectively, in (\ref{3.1}), we obtain
\begin{equation}\nonumber
\left \| 2^{3n}f\left( \left\langle \frac{x}{2^{n}},\frac{y}{2^{n}}\right\rangle ^{*}\frac{z}{2^{n}}\right) -\left\langle 2^{n}f\left( \frac{x}{2^{n}}\right) ,2^{n}f\left( \frac{y}{2^{n}}\right)  \right\rangle ^{*} 2^{n}f\left( \frac{z}{2^{n}} \right) \right \| \leq 2^{3n} \varphi \left( \frac{x}{2^{n}},\frac{y}{2^{n}},\frac{z}{2^{n}}\right) 
\end{equation}
for all $x,y,z \in \Xi$. Taking the limit as $n \rightarrow \infty$. using (\ref{2.15}) and (\ref{2.17}), we obtain
\begin{equation}\nonumber
H(\left\langle x,y\right\rangle ^{*}z)=\left\langle H(x),H(y)\right\rangle ^{*}H(z)
\end{equation}
for all $x,y,z \in \Xi$. So $H$ is a Hilbert $C^{*}$-module $\ast$-homomorphism.
\end{proof}

The following corollary gives us the Hyers-Ulam-Rassias stability of $\ast$-homomorphisms in $\Xi$.

\begin{cor}
Let $p>3$, $\theta$ be a nonnegative real number and $f:\Xi \rightarrow \nabla$ be a mapping  satisfying (\ref{2.10}), (\ref{2.11}) and (\ref{3.4}). Then there exists a unique Hilbert $C^{*}$-module $\ast$-homomorphism $H:\Xi\rightarrow \nabla$ such that 
\begin{equation}
\|f(x)-H(x)\| \leq \frac{2\theta}{2^{p}-2^{3}}\|x\|^{p} 
\end{equation}
for all $x \in \Xi$.
\end{cor}

\begin{proof}
Take $\varphi (x,y,z)= \theta (\|x\|^{p}+\|y\|^{p}+\|z\|^{p})$ and $L=2^{3-p}$. Then the proof follows from Theorem \ref{theorem 4}.
\end{proof}

\section{Conclusion}

In this paper, we introduced the idea of $\ast$-homomorphisms on Hilbert $C^{*}$-modules. Furthermore, we established the Hyers-Ulam and Hyers-Ulam-Rassias stability of Hilbert $C^{*}$-module homomorphisms and Hilbert $C^{*}$-module $\ast$-homomorphisms using the fixed point method.

\medskip

\section*{Declarations}

\medskip

\noindent \textbf{Availability of data and materials}\newline
\noindent No datasets were generated or analyzed during the current study.

\medskip

\noindent \textbf{Human and animal rights}\newline
\noindent We would like to mention that this article does not contain any studies
with animals and does not involve any studies over human being.

\medskip

\noindent \textbf{Conflict of interest}\newline
\noindent The authors declare that they have no competing interests.

\medskip

\noindent \textbf{Fundings} \newline
\noindent  Not applicable.

\medskip

\noindent \textbf{Authors' contributions}\newline
\noindent The authors equally conceived of the study, participated in its
design and coordination, drafted the manuscript, participated in the
sequence alignment, and read and approved the final manuscript. 

\medskip

{\footnotesize

\end{document}